\documentclass[12pt]{amsproc}
\newtheorem{theorem}{\sc Theorem}[section]
\newtheorem{lemma}[theorem]{\sc Lemma}

\newtheorem{corollary}[theorem]{\sc Corollary}

\newtheorem{conjecture}[theorem]{\sc Conjecture}

\newcommand{\al }{\alpha }
\usepackage{amssymb, amsthm, amsmath}
\renewcommand{\>}{\rangle}
\newcommand{\<}{\langle}

\begin{document}
\author{Jo\~ao Azevedo}
\address{Department of Mathematics, University of Brasilia\\
Brasilia-DF \\ 70910-900 Brazil}
\email{J.P.P.Azevedo@mat.unb.br}

\author{Pavel Shumyatsky}
\address{Department of Mathematics, University of Brasilia\\
Brasilia-DF \\ 70910-900 Brazil}
\email{pavel@unb.br}
\thanks{Supported by CNPq and FAPDF}

\title[Conciseness of some words]{On finiteness of verbal subgroups}
 \begin{abstract} Given a group-word $w$ and a group $G$, the set of $w$-values in $G$ is denoted by $G_w$ and the verbal subgroup $w(G)$ is the one generated by $G_w$. The word $w$ is concise if $w(G)$ is finite for all groups $G$ in which $G_w$ is finite. We obtain several results supporting the conjecture that the word $[u_1,\dots,u_s]$ is concise whenever the words $u_1,\dots,u_s$ are non-commutator.
\end{abstract}
\maketitle

\section{Introduction}

If $w=w(x_1,\dots,x_k)$ is a group-word in variables $x_1,\dots,x_k$,
we denote by $w(G)$ the verbal subgroup of the group $G$ generated by the set $G_w$
of all values of $w$ in $G$. A word $w$ is concise if $w(G)$ is finite for every group $G$ such that $G_w$ is finite. P. Hall had conjectured that every word is concise but the conjecture was refuted by Ivanov in \cite{ivanov}. Further examples of words which are not concise were produced in \cite{olsh}.

On the other hand, many words are known to be concise. P. Hall  proved conciseness of non-commutator words (i.e., a word such that the sum of the exponents of some variable involved in
it is non-zero). In \cite{ts} Turner-Smith showed that the derived words are concise, and J.C.R. Wilson \cite{jcrw} subsequently extended this result to all multilinear commutator words. Multilinear commutator words, also called outer commutators, are words obtained by nesting commutators, but using always different variables. Another proof of conciseness of multilinear commutators was obtained in \cite{femo}. It was proved in \cite{femotra} and, independently, in \cite{abru} that the $n$th Engel word $[x,_ny]$ is concise for $n\leq4$. The main result of the recent paper \cite{detotoshu} says that any word that can be written as a commutator of two non-commutator words is concise.  In view of the latter result the following conjecture is natural.

\begin{conjecture}\label{conj} The word $[u_1,\dots,u_s]$ is concise whenever the words $u_1,\dots,u_s$ are non-commutators.
\end{conjecture}

{\sc Convention.} Throughout the paper, and in Conjecture \ref{conj}, whenever we deal with a commutator word $[u,v]$ it is assumed that the variables involved in the word $u$ do not occur in $v$, and those involved in $v$ do not occur in $u$.

The purpose of this paper is to furnish evidence for the above conjecture. In particular we show that the conjecture holds if $s=3$:
\begin{theorem}\label{3 non-commutators} Let $u_1,u_2,u_3$ be non-commutator words. Then the word $[u_1,u_2,u_3]$ is concise.
\end{theorem}

Further, we show that Conjecture \ref{conj} is valid in when $u_1=\dots=u_s$. More precisely, we establish the following theorem.

\begin{theorem}\label{engel-like non-commutator}
Let $u=u(x_1,\dots,x_k)$ and $v=v(y_1,\dots,y_l)$ be non-commutator words. For $i=1,2,\dots,s$ set $u_i=u(x_{1i},\dots,x_{ki})$. Then both words $[u_1,\dots,u_s]$ and $[v,u_1,\dots,u_s]$ are concise.
\end{theorem}

In recent years there has been considerable interest in the question whether all words are concise in profinite groups (cf. Segal \cite[p.\ 15]{Segal} or Jaikin-Zapirain \cite{jaikin}). An equivalent form of this question is whether all words are concise in the class of residually finite groups. Recall that a word $w$ is concise in a class of groups $\mathcal X$ if $w(G)$ is finite in every group $G\in\mathcal X$ such that $G_w$ is finite. Currently, there are no known examples of profinite groups $G$ admitting a word $w$ such that $G_w$ is finite while $w(G)$ is infinite. On the other hand, there are words that are  concise in profinite groups while their conciseness in the class of all groups is still unknown (see \cite{dms1,dms2,dms3} and references therein). 

A word $w$ is boundedly concise in a class of groups $\mathcal X$ if for every integer $m$ there exists a number $\nu=\nu(\mathcal X,w,m)$ such that whenever $|G_w|\leq m$ for a group $G\in\mathcal X$ it follows that $|w(G)|\leq\nu$. Every word which is concise in the class of all groups is actually boundedly concise \cite{femo} (however it is unclear whether every word which is concise in profinite groups is boundedly concise).

Our next result provides a useful tool for proving conciseness of words in profinite (or residually finite) groups.

\begin{theorem}\label{[concise, non-commutator]}
Let $u$ be a word which is (boundedly) concise in profinite groups and $v$ a non-commutator word. The word $[u,v]$ is (boundedly) concise in profinite groups. 
\end{theorem}

An immediate corollary of the above result is that Conjecture \ref{conj} is valid for profinite groups in the following stronger form.

\begin{corollary}\label{coro} The word $[u_1,\dots,u_s]$ is boundedly concise in profinite groups whenever the words $u_1,\dots,u_s$ are non-commutator.
\end{corollary}

Indeed, we know that any non-commutator word is boundedly concise (in the class of all groups). By induction on $s$ assume that the word $[u_1,\dots,u_{s-1}]$ is boundedly concise in profinite groups. Thus, an application of Theorem \ref{[concise, non-commutator]} proves the result.

It should be noted that the techniques developed in this paper can be used to establish conciseness of some other words. In fact, the results obtained here indirectly suggest that perhaps every commutator word $[u,v]$ is concise whenever the words $u$ and $v$ are.
This is exemplified in our next result. 
\begin{theorem}\label{mcw, non-commutator}
Let $u$ be a multilinear commutator word and $v$ a non-commutator word. The word $[u,v]$ is concise.
\end{theorem}

In the next section we collect some helpful lemmas that will be needed in the proofs of the main theorems. Section 3 contains a proof of Theorem \ref{engel-like non-commutator}. In Section 4 we establish Theorem \ref{3 non-commutators}. Section 5 deals with Theorem \ref{mcw, non-commutator}. In Section 6 we handle profinite groups and prove Theorem \ref{[concise, non-commutator]}.

\section{Preliminaries}

The first lemma is straightforward. Here and throughout the article we use the expression ``$(a,b,\dots)$-bounded" to mean that a quantity is bounded by a certain number depending only on the parameters $a,b,\dots$. 
\begin{lemma}\label{normal abelian}
Let $A$ be an abelian normal subgroup of a group $G$. Then $[ab,c]=[a,c][b,c]$ for all $a,b\in A$ and $c\in G$. In particular, if $n$ is an integer, $[a^n,c]=[a,c]^n$.
\end{lemma}

The next result is well known (see for example \cite[Lemma 4]{dms1}).

\begin{lemma}\label{w(G)' finite}
Let $w$ be a group-word and let $G$ be a group in which $w$ takes only finitely many, say $m$, values. Then the commutator subgroup $w(G)'$ is finite with $m$-bounded order.
\end{lemma}

The following lemma will often be used without explicit references.
\begin{lemma}\label{aaa} Let $w$ be a group-word and let $G$ be a group in which $w$ takes only finitely many values. Suppose that $N$ is a normal torsion subgroup of $G$. Then $w(G)$ is finite if and only if so is $w(G)N/N$.
\end{lemma}
\begin{proof} This is straightforward using the fact that in view of Lemma \ref{w(G)' finite} the intersection $N\cap w(G)$ is finite.
\end{proof}

The next elementary lemma furnishes a useful tool for handling of non-commutator words. For the reader's convenience we provide a proof.
\begin{lemma}\label{non-commutator > x^n}
Let $w=w(x_1,\dots,x_k)$ be a non-commutator word. There is $n\geq1$ such that in any group $G$ all values of the word $x^n$ are $w$-values.
\end{lemma}
\begin{proof}
As $w$ is non-commutator, the sum of the exponents of some variable in $w$ is non-zero. Let $x_i$ be that variable and $n$ the sum. Given a group $G$, if we replace $x_i$ by an arbitrary element $g\in G$ and $x_j$ by 1 for $j\neq i$, then the word $w$ takes the value $g^n$ and we conclude that all $n$th powers in $G$ are $w$-values. 
\end{proof}

As usual, we denote by $G^n$ the subgroup of a group $G$ generated by the $n$th powers.

\begin{lemma}\label{[u(G),G^n] = 1} Let $w=[u,v]$ for some word $u$ and a non-commutator word $v$. Assume that $G$ is a group in which the set $G_w$ is finite. There exists a positive integer $n$, depending only on $v$ and $|G_w|$, such that $[u(G),G^n] = 1$.
\end{lemma}

\begin{proof} Let $|G_w| = m$. By Lemma \ref{non-commutator > x^n} there is a positive integer $e$ such that all $e$th powers are also $v$-values in $G$. Choose $x\in G$, and let $a \in G_u$. The elements $[a,x^{re}]$ are $w$-values for all choices of $r$. Since $|G_w| = m$, for some integers $0 \leq r < s \leq m$ we must have $[a, x^{re}] = [a, x^{se}]$. This implies that $x^{(s-r)e}$ centralizes $a$, and we conclude that for every $a \in G_u$ and $x \in G$, there is a $t$, with $1 \leq t \leq m$, such that $[a, x^{te}] = 1$. Set $n = m!e$. We see that $x^n$ must centralize all $u$-values of $G$. Therefore $[u(G), G^n] = 1$.
\end{proof} 

\begin{lemma}\label{[W, W_2, ..., W_k] = 1} Let $w_1, \dots, w_k$ be non-commutator words and let $w = [w_1, \dots, w_k]$. Assume that $G$ is a group such that $G_w$ is finite. For $i=1,\dots,k$, let $V_i$ be the subgroup obtained by formally replacing $w_i(G)$ with $w(G)$ in $[w_1(G),\dots, w_k(G)]$. Then the subgroups $V_i$ are finite.
\end{lemma}

\begin{proof}
We will establish the finiteness of the subgroup $$V_1=[w(G), w_2(G), \dots, w_k(G)],$$ as  the other cases are analogous. In view of Lemma \ref{w(G)' finite} we can pass to the quotient  $G/w(G)'$ and assume that $w(G)$ is a finitely generated abelian subgroup. Let $n$ be a positive integer such that every $n$th power is a $w_1$-value in $G$. Choose $a\in G_w$ and $g_i \in G_{w_i}$ for $i = 2, \dots, k$. Observe that $[a, g_2, \dots, g_k]^{rn} = [a^{rn}, g_2, \dots, g_k]$, by Lemma \ref{normal abelian}, and the choice of $n$ guarantees that this is always a $w$-value. Let $|G_w| = m$. It follows that the order of $[a,g_2,\dots,g_k]$ divides $m!n$. As $w(G)$ is finitely generated and abelian, the subgroup $K$ of $w(G)$ consisting of all elements whose orders divide $m!n$ is a finite normal subgroup of $G$. For every choice of $a,g_2,\dots,g_k$ the element $[a,g_2,\dots,g_k]$ belongs to $K$. Therefore $V_1\leq K$, and so $V_1$ must be finite. 
\end{proof}

The next result provides a sufficient condition for conciseness of certain words. Throughout, $\langle X \rangle$ denotes the subgroup generated by a set $X$.

\begin{lemma}\label{first reduction argument}
Let $w=[u,v]$ for some word $u$ and a non-commutator word $v$. Let $G$ be a group in which the set $G_w$ is finite. Suppose that $w(G)^j$ centralizes $v(G)$ for some positive integer $j$. Then $w(G)$ is finite. 
\end{lemma}

\begin{proof} It follows from Lemma \ref{w(G)' finite} that the elements of finite order in $w(G)$ form a finite normal subgroup containing $w(G)'$. We can pass to the quotient over this subgroup and simply assume that $w(G)$ is torsion-free abelian.

Let $g$ be a $v$-value in $G$ and consider the subgroup $K = w(G) \langle g \rangle$. As $w(G)^j$ commutes with $g$, it is a central subgroup of $K$. Let $n$ be as in Lemma \ref{[u(G),G^n] = 1}. Since $g^n$ commutes with $u(G)$, it also commutes with $w(G)$. Therefore $w(G)^j \langle g^n \rangle$ is a central subgroup of finite index of $K$. By Schur's Theorem (\cite[10.1.4]{robinson}), $K'$ is finite. Since $K'$ is a subgroup of $w(G)$, which is torsion-free, $K'$ must be trivial and so $K$ is abelian. Let $n$ be as in Lemma \ref{[u(G),G^n] = 1} and choose $h\in G_u$. As both $g, g^h$ belong to $K$, we have $$[h,g]^n =[h, g^n] = 1.$$ Since $h$ and $g$ were chosen arbitrarily, conclude that all $w$-values of $G$ have finite orders. Hence, $w(G)$ is finite.
\end{proof}

The next result is immediate from \cite[Lemma 10]{dms1}.

\begin{lemma}\label{p-lemma}
Let $w = w(x_1, x_2, \dots, x_k)$ be a word, and let $G$ be a nilpotent group of class $c$ generated by $k$ elements $a_1, a_2,\dots,a_k$. Denote by $X$ the set of all conjugates in $G$ of elements of the form $w(a_1^i, a_2^i, \dots, a_k^i)$, where $i$ is an integer, and assume that $|X|=m$ for some integer $m$. Then the subgroup $\langle X \rangle$ has finite $(c,m)$-bounded order.
\end{lemma}

\section{Proof of Theorem \ref{engel-like non-commutator}} 

If $x$ is an element of a group $G$ and $Y$ is a subset of $G$, we write $x^Y$ to denote the set of conjugates of $x$ by elements of $Y$. We say that a subset of $G$ is normal if it is invariant under every inner automorphism of $G$.

\begin{lemma}\label{finite K-conjugation classes} Let $G$ be a group, $Y$ a normal subset of $G$, and let $K = \langle Y \rangle$. Suppose that $x\in G$ is an element such that $x^{Y}$ is  finite. Assume further that for every $y\in Y$ there is a positive integer $m=m(y)$ such that $[x^K,y^m]=1$. Then the set $x^K$ is finite.
\end{lemma} 
\begin{proof} Write $x^Y=\{x^{y_1},\dots, x^{y_k}\}$ for $y_1,\dots,y_k\in Y$. Any element $g\in K$ can be written as a product of elements from $Y\cup Y^{-1}$. Since for every $y\in Y$ there is a positive integer $m$ such that $[x^K,y^m]=1$, it follows that $g$ can be written as $g=cg_1\dots g_r$, where $c\in C_G(x^K)$ and $g_1,\dots,g_r\in Y$. 

We claim that $x^{g} = x^{y_{i_1} y_{i_2} \dots y_{i_r}}$ for some $y_{i_j}$ in $\{{y_1},\dots,{y_k}\}$. If $r = 1$, the result is clear so assume that $r\geq2$. We know that $x^{g_1}=x^{y_i}$ for some $1 \leq i \leq k$. Therefore we can write $x^g$ as $$x^{g_1g_2\dots g_r} = x^{y_ig_2\dots g_r} = x^{g_2'\dots g_r' y_i},$$ where $g_j' = y_i g_j y_i^{-1}$ also belongs to $Y$ since $Y$ is a normal subset of $G$. The induction hypothesis applied to $x^{g_2'\dots g_r'}$ establishes the claim. 

Now, define a well-order in the set of all formal products of the elements $\{y_1, y_2, \dots, y_k\}$ by the following rule: we say that $y_{i_1}y_{i_2}\dots y_{i_r}$ $<$ $y_{i'_1}y_{i'_2}\dots y_{i'_l}$ provided that $r < l$ or $r = l$ and there exists a positive integer $1 \leq t \leq r$ such that $i_{t} < i'_{t}$ and $i_n = i_n'$ for all $t < n \leq r$. As we have seen above, for any $g\in K$ there are $y_{i_1} y_{i_2} \dots y_{i_r}\in Y$ such that $x^g$ equals $x^{y_{i_1} y_{i_2} \dots y_{i_r}}$. We can choose the product $y_{i_1} y_{i_2} \dots y_{i_r}$ to be the smallest, in the sense of the order relation defined above. We claim that $i_1 \geq i_2 \geq \dots \geq i_r$. Assume that $i_{k}<i_{k+1}$ for some $k$. Then the product $y_{i_1}  \dots y_{i_{k-1}}y_{i_k}y_{i_{k+1}}y_{i_{k+2}} \dots y_{i_r}$ can be rewritten in the form $y_{i_1}  \dots y_{i_{k-1}}s y_{i_k}y_{i_{k+2}} \dots y_{i_r}$, where $s = y_{i_{k }}y_{i_{k+1}}y_{i_{k}}^{-1}$ also belongs to $Y$. From the previous paragraph we know that $x^{y_{i_1}  \dots y_{i_{k-1}}s}  = x^{y_{i'_1}  \dots y_{i'_{k-1}}y_{i'_k}}$, where the $y_{i_j'}$ belong to $\{y_1, y_2, \dots, y_k\}$. Since $y_{i'_1}  \dots y_{i'_{k-1}}y_{i'_k}y_{i_k}y_{i_{k+2}} \dots y_{i_r}$ is smaller than $y_{i_1}  \dots y_{i_k}y_{i_{k+1}}y_{i_{k+2}} \dots y_{i_r}$, we obtain a contradiction. We conclude that $x^g$ equals $x^{y_{i_1} y_{i_2} \dots y_{i_r}}$, with decreasing indexes. The fact that some power of $y_{i}$ centralizes $x^K$ implies that each $y_i$ occurs only finitely many times. Hence, $|x^K|$ is finite.\end{proof}

Recall that an FC-group is a group in which all conjugacy classes are finite. The next lemma is well-known (see for example \cite[Theorem 4.32]{robinson-finiteness-I}).

\begin{lemma}\label{G' periodic in FC-groups}
Let $G$ be an FC-group. The elements of finite order in $G$ form a subgroup containing $G'$.
\end{lemma}

An element $a$ of a group $G$ is called a (left) Engel element if for any $b\in G$ there exists $k=k(a,b)\geq 1$ such that $[b,_ka]=1$. Here we use the abbreviation $[b,\,{}_ka]:=[b,a,a,\dots,a]$ where $a$ is repeated $k$ times. The element $a$ is $k$-Engel if $[b,_ka]=1$ for all $b\in G$. The following result is due to Gruenberg (see \cite{grue} or \cite[Theorem 12.3.3]{robinson}). 

\begin{lemma}\label{gruenberg}
Let $G$ be a soluble group generated by finitely many Engel elements. Then $G$ is nilpotent.
\end{lemma}

The following lemma will be helpful (cf \cite[Lemma 9]{dms1}).

\begin{lemma}\label{semidirect product nilpotence}
Let $G=H\langle a \rangle$ be a product of a nilpotent normal subgroup $H$ and a cyclic subgroup $\<a\>$. If $[H,_na]=1$, then $G$ is nilpotent.
\end{lemma}

It will be convenient to prove Theorem \ref{engel-like non-commutator} first in the special case where $w = [x_1^n, \dots, x_s^n]$. 
\begin{lemma}\label{powers}
Let $n$ and $s$ be positive integers and let $G$ be a group where the word $[x_1^n, \dots, x_s^n]$ takes only  finitely many values. Then $w(G)$ is finite. 
\end{lemma}
\begin{proof}
Denote the word $[x_1^n, \dots, x_{s-1}^n]$ by $u_0$ and the word $x_s^n$ by $v_0$.  Let $a$ be a $u_0$-value in $G$ and $b$ a $v_0$-value. The set $$a^{G_{v_0}} = \{a[a,b] \, | \, b \in G_{v_0}\}$$ is finite and moreover $a^{G_{v_0}}\subseteq  aG_w$. Note that $G_{v_0}$ is a normal subset of $G$ and, by Lemma \ref{[u(G),G^n] = 1}, for some positive integer $e$ we have $[u_0(G),G^e]=1$.  Lemma \ref{finite K-conjugation classes} guarantees that $a^{v_0(G)}$ is finite. In particular, we also have that $a^{u_0(G)}$ is finite, since $u_0(G) \leq v_0(G)$. As the $u_0$-values generate $u_0(G)$ and have finite $u_0(G)$-conjugacy classes, $u_0(G)$ is an FC-group. Applying Lemmas \ref{G' periodic in FC-groups} and \ref{aaa}, we can pass to the quotient $G/u_0(G)'$ and assume that $u_0(G)$ is abelian. Moreover, we can pass to the quotient over the maximal torsion subgroup of $u_0(G)$ and without loss of generality assume that $u_0(G)$ is torsion-free, again by Lemma \ref{aaa}. 

Choose $g_1,\dots,g_s\in G$ with $g_1\in u_0(G)$. Since $u_0(G)$ is abelian, by Lemma \ref{normal abelian} for any integer $r$ we have $[g_1,g_2^n, \dots,g_s^n]^{rn}=[g_1^{rn},g_2^n, \dots, g_s^n]$ $\in G_w$.  The finiteness of the set of $w$-values in $G$, implies that the element $[g_1,g_2^n, \dots,g_s^n]$ has finite order. This happens for every choice of $g_1,\dots,g_s$. As $w(G)$ is torsion free, $[g_1,g_2^n, \dots,g_s^n] = 1$, and so $v_0(G)=G^n$ is nilpotent of class at most $2s-2$. We are now in a position to apply Lemma \ref{p-lemma}. Take arbitrary elements $a_1, \dots, a_s\in G$ and write $b_i$ for $a_i^n$. Observe that $\langle b_1, b_2, \dots, b_s \rangle$ is a nilpotent subgroup of $G$. Let $\gamma$ denote the word $[y_1,y_2,\dots,y_s]$. Clearly, $\gamma(b_1^i, b_2^i,\dots,b_s^i)=[a_1^{in}, a_2^{in}, \dots, a_s^{in}]$ is a $w$-value for any given integer $i$. Therefore the set $$\{\gamma(b_1^i, b_2^i, \dots, b_s^i) \, | \, i\in\Bbb Z\}$$ is finite. Lemma \ref{p-lemma} tells us that $[b_1,\dots,b_s]$ has finite order. As $w(G)$ is torsion-free, conclude that $w(G)=1$. This completes the proof of the theorem for the word $[x_1^n,\dots,x_s^n]$.
\end{proof}

We can now prove Theorem \ref{engel-like non-commutator}, which we restate here for the reader's convenience.

\medskip
 {\sc Theorem \ref{engel-like non-commutator}.} {\it Let $u=u(x_1,\dots,x_k)$ and $v=v(y_1,\dots,y_l)$ be non-commutator words. For $i=1,2,\dots,s$ set $u_i=u(x_{1i},\dots,x_{ki})$. Then both words $[u_1,\dots,u_s]$ and $[v,u_1,\dots,u_s]$ are concise.}

\begin{proof} Let the word $w$ be either $[u_1,\dots,u_s]$ or $[v,u_1,\dots,u_s]$, and let $G$ be a group in which $G_w$ is finite. We need to show that $w(G)$ is finite, too. By virtue of of Lemmas \ref{w(G)' finite} and \ref{aaa} without loss of generality we may assume that $w(G)$ is torsion-free abelian.

Suppose first that $w$ is the word $[u_1, u_2, \dots, u_s]$, where $u$ is an arbitrary non-commutator word. Note that $w(G)=\gamma_s(u(G))$ is the $s$th term of the lower central series of $u(G)$. Since $w(G)$ abelian, $u(G)$ is soluble. Let $a$ be any $u$-value in $G$. By Lemma \ref{[W, W_2, ..., W_k] = 1} we have  $[w(G),_{s-1}a]=1$. For any $g\in G$ the element $[g,a]$ belongs to $u(G)$ and $[g,_sa]$ belongs to $w(G)$. It follows that $[g,_{2s-1}a]=1$, and this shows that all $u$-values are $(2s-1)$-Engel in $G$. Lemma \ref{gruenberg} now implies that any subgroup of $G$ generated by finitely many $u$-values is nilpotent. In particular, $u(G)$ is locally nilpotent. The set of elements of finite order in a locally nilpotent group is a torsion subgroup. So $u(G)$ contains a characteristic torsion subgroup $L$ such that $u(G)/L$ is torsion-free. Thus, we pass to the quotient $G/L$ and assume that $u(G)$ is torsion-free.

According to Lemma \ref{non-commutator > x^n} there is a positive integer $n$ such that $g^n$ is a $u$-value for all $g\in G$. All values of the word $w_0 = [x_1^n, x_2^n, \dots, x_s^n]$ are $w$-values, so $|G_{w_0}|$ is finite. We conclude that $w_0(G)$ is finite, by Lemma \ref{powers}. Since $w(G)$ is torsion-free, $w_0(G)=1$. Thus, $G^n$ is nilpotent of class at most $s-1$.

Let $\beta$ be the word $[u_1, u_2, \dots, u_{s-1}]$. We can rewrite the word $[u_1, u_2, \dots, u_{s-1}, x^n]$ as $(x^{-n})^\beta x^n$. Let $g$ be an $n$th power in $G$ and $h$ a $\beta$-value. Observe that the subgroup $\langle g,g^h \rangle$ is nilpotent. If $\al = \al(x,y)$ is the word $x^{-1}y$, note that the $\al$-values $$\al((g^h)^i, g^i) = (g^h)^{-i}g^{i} = [h, g^i]$$ are $w$-values for every $i$. Lemma \ref{p-lemma} ensures that all elements of the form $[h,g]$ are torsion. Since $u(G)$ is torsion-free, conclude that $[u_1(G),u_2(G),\dots,u_{s-1}(G),G^n]=1$.

Now choose $a\in G_u$ and consider the subgroup $K=\langle a,a^h\rangle$, where $h$ is as above. Since $a^n$ commutes with $h$, it follows that $a^n\in Z(K)$ and so $K$ is central-by-finite. According to Schur's theorem $K'$ is finite. Taking into account that $u(G)$ is torsion-free, we deduce that $K$ is abelian. Now we have $(a^{-1}a^h)^n=a^{-n}(a^n)^h=1$. Since $u(G)$ is torsion-free, it follows that $a=a^h$. This happens for every $u$-value $a$ and every $\beta$-value $h$. Therefore $w(G)=1$. This completes the proof in the case where $w$ is the word $[u_1, u_2, \dots, u_s]$.

Now, assume that $w = [v, u_1, u_2, \dots, u_s]$. Let us show that $T=u(G)\cap v(G)$ is nilpotent. Choose $u$-values $a_1, \dots, a_{2s}$ and a $v$-value $b$. Note that $[b,a_1, \dots,a_s]$ is a $w$-value and therefore, in view of Lemma \ref{[W, W_2, ..., W_k] = 1}, we have $[b,a_1,\dots,a_{2s}]=1$, as $w(G)$ is torsion-free. Let $t_1, t_2, \dots$, $t_{2s+1}$ be arbitrary elements in $T$. Since $t_1$ belongs to $v(G)$, it can be written as a product of finitely many $v$-values. Similarly, for $2 \leq i \leq 2s+1$ each $t_i$ is as a product of $u$-values. The commutator $[t_1,t_2,\dots,t_{2s+1}]$ equals the product of finitely many commutators of the form $[b,a_1,\dots,a_{2s}]$, where $b\in G_v$ and $a_i\in G_u$. All those commutators are trivial and this proves that $T$ is nilpotent. 

Let $a \in G_u$ and note that $[w(G), _s a] = 1$. As $T\leq v(G)$, the subgroup $[T,_sa]$ is contained in $[v(G),_sa]$, which is contained in $w(G)$. Therefore $[T, _{2s}a]=1$, that is, $a$ acts on $T$ as a left $2s$-Engel element. Set $K_0=T\langle a\rangle$. By Lemma \ref{semidirect product nilpotence} the subgroup $K_0$ is nilpotent. We remark that $w(G)$ is contained in $T$. Let $d$ be a value of the word $[v,u_1,\dots,u_{s-1}]$. Since $a$ and $[d,a]$ belong to $K_0$, both $a$ and $a^d$ belong to $K_0$ and so the subgroup $S=\langle a,a^d\rangle$ is nilpotent. Lemma \ref{[u(G),G^n] = 1} tells us that there is a positive integer $n$  such that $a^n$ commutes with $d$. Observe that $a^n\in Z(S)$ and so $S$ is central-by-finite. According to Schur's theorem the commutator subgroup $S'$ is finite. Note that $S'\leq w(G)$. Since $w(G)$ is torsion-free, we conclude that $S$ is abelian. It follows that the $w$-value $a^{-d}a$ has finite order dividing $n$ since $(a^{-d}a)^n=a^{-nd}a^n=[d,a^n]=1$. Taking into account that $w(G)$ is torsion-free deduce that $[d,a]=1$ and so $w(G)=1$. The proof is complete. 
\end{proof}

\section{Proof of Theorem \ref{3 non-commutators}} 

Let $u_1, u_2$, and $u_3$ be non-commutator words, and let $w = [u_1, u_2, u_3]$. Then of course $w(G)=[u_1(G),u_2(G),u_3(G)]$. Let $G$ be a group in which $|G_w|=m$ is finite. We want to show that $w(G)$ is finite. Because of Lemmas \ref{w(G)' finite} and \ref{aaa} without loss of generality we may assume that $w(G)$ is torsion-free abelian. 

\begin{lemma}\label{[W_1, W_2, W] = 1} 
Let $A$ be an abelian normal subgroup of $G$, and assume that $A$ is contained in $u_i(G)$ for some $i$. The subgroup obtained by replacing $u_i(G)$ with $A$ in $[u_1(G),u_2(G),u_3(G)]$ is trivial. 
\end{lemma}

\begin{proof}
We only deal with the subgroup $[A,u_2(G),u_3(G)]$, as the other results follow in a similar way. The word $u_1$ is a non-commutator so, by Lemma \ref{non-commutator > x^n}, we have $G^n\leq u_1(G)$ for some $n$. Let $a\in A$, $g_2\in G_{u_2}$, and $g_3\in G_{u_3}$. The commutator $[a, g_2, g_3]^{rn}$ equals $[a^{rn}, g_2, g_3]$ for any integer $r$. This is a $w$-value. As $G_w$ is finite, we conclude that $[a,g_2, g_3]$ has finite order. Since $w(G)$ is torsion-free, $[a, g_2, g_3]=1$. Hence the result.
\end{proof} 

\begin{lemma}\label{[W_1, W_2, G^n] = 1}
There exists a positive integer $n$ such that the subgroups $$[G^n,u_2(G), u_3(G)], [u_1(G), G^n, u_3(G)],\text{ and } [u_1(G), u_2(G), G^n]$$ are trivial. Moreover, the subgroup $G^n$ is contained in $u_i(G)$ for each $i=1,2,3$.
\end{lemma}
\begin{proof}
By Lemma \ref{non-commutator > x^n}, there exists a natural number $n_i$ such that $g^{n_i}$ is a $u_i$-value, for every $g\in G$ and $i=1,2,3$. Set $n=n_1n_2n_3$ and observe that the values of the word $x^n$ are contained $G_{u_1}\cap G_{u_2}\cap G_{u_3}$.

In particular, the values of the word $w_0 = [x_1^n, x_2^n, x_3^n]$ are contained in $G_w$. In view of Theorem \ref{engel-like non-commutator} we conclude that $w_0(G)$ is finite. Since $w(G)$ is torsion-free, $w_0(G)=1$. Thus, $G^n$ is nilpotent of class at most 2. Write $$[x^n, u_2, u_3] = [x^{-n}(x^n)^{u_2}, u_3] = (x^{-n})^{u_2}x^n(x^{-n})^{u_3}(x^n)^{u_2u_3}.$$ Let $\alpha(y_1, y_2, y_3, y_4)$ be the group word $y_1^{-1}y_2y_3^{-1}y_4$. Choose $g_2\in G_{u_2}$, $g_3\in G_{u_3}$, and $t\in G_{x^n}$. Let $K = \langle t, t^{g_2}, t^{g_3}, t^{g_2g_3}\rangle$. Note that the set $$\{\alpha((t^{g_2})^i, t^i, (t^{g_3})^i, (t^{g_2g_3})^i) \, | \, i\in\Bbb Z\}$$ is a subset of $G_w$, as $$\alpha((t^{g_2})^i, t^i, (t^{g_3})^i, (t^{g_2g_3})^i) = [t^{ni}, g_2, g_3].$$ Applying Lemma \ref{p-lemma}, we conclude that $K$ has finite order. Since $w(G)$ is torsion-free, $[G^n,u_2(G),u_3(G)]=1$. The triviality of the other subgroups can be established in a similar way.
\end{proof}

Throughout the rest of this section the word $[u_1, u_2]$ will be denoted by $v$.

\begin{lemma}\label{second reduction argument}
Suppose that the image of $w(G)$ in $G/v(G)'$ is finite. Then $w(G)=1$. 
\end{lemma}
\begin{proof}
By Lemma \ref{[W, W_2, ..., W_k] = 1}, the subgroup $[u_1(G), u_2(G), w(G)]$ is finite. Since $w(G)$ is torsion-free, it follows that $[u_1(G), u_2(G), w(G)]=1$. Note that $$[u_1(G), u_2(G), w(G)] = [v(G), [v(G), u_3(G)]].$$ The Three Subgroup Lemma yields $$[v(G),v(G),u_3(G)] \leq [v(G), u_3(G), v(G)] = 1$$ and so $v(G)'$ centralizes $u_3(G)$. As the image of $w(G)$ in $G/v(G)'$ is finite, the subgroup $w(G)^j$ is contained in $v(G)'$, for some positive integer $j$. Therefore $[w(G)^j, u_3(G)] = 1$. In view of Lemma \ref{first reduction argument} deduce that $w(G)$ is finite. Taking into account that $w(G)$ is torsion-free, conclude that $w(G)=1$.
\end{proof}

Now we are in position to prove Theorem \ref{3 non-commutators}. For the reader's convenience, we restate it here.

{\sc Theorem \ref{3 non-commutators}. }{\it  Let $u_1,u_2,u_3$ be non-commutator words. Then the word $[u_1,u_2,u_3]$ is concise.}

\begin{proof}
We keep all the assumptions made in this section. Lemma \ref{second reduction argument} allows us to additionally assume that $v(G)$ is abelian. By Lemma \ref{[W_1, W_2, W] = 1}, the subgroups $$[v(G), u_2(G), u_3(G)] \quad \text{and} \quad [u_1(G), v(G), u_3(G)]$$ \noindent are trivial. 

Choose $g_1\in G_{u_1}$, $g_2 \in G_{u_2}$ and $g_3\in G_{u_3}$. Let $C = C_G(u_3(G))$ and set $\overline{G}=G/C$. We use the bar notation to denote the images in $\overline{G}$. Since all values of the words $[u_1,u_2,u_1]$ and $[u_1,u_2,u_2]$ centralize $G_{u_3}$, in $\overline{G}$ we have $$[\bar{g}_1, \bar{g}_2,\bar{g}_1] = [\bar{g}_1, \bar{g}_2, \bar{g}_2] = \bar{1}.$$ We see that the subgroup $\langle \bar{g}_1, \bar{g}_2 \rangle$ is nilpotent of class at most 2. It follows that $[\bar{g}_1, \bar{g}_2]^j$ equals $[\bar{g}_1^j, \bar{g}_2]$ for any integer $j$. By Lemma \ref{[W_1, W_2, G^n] = 1}, there exists some $n$ such that the subgroup $[G^n, u_2(G)]$ also centralizes $u_3(G)$. Then $[\bar{g}_1^n, \bar{g}_2] = \bar{1}$ and so $[g_1, g_2]^n$ centralizes $u_3(G)$. 

Recall that $v(G)$ is abelian. The $w$-value $[g_1,g_2,g_3]$ can be written as $[g_1,g_2]^{-1}[g_1, g_2]^{g_3}$. We have  $$[g_1,g_2,g_3]^n=([g_1,g_2]^{-1}[g_1,g_2]^{g_3})^n=[g_1,g_2]^{-n}([g_1,g_2]^n)^{g_3}=[[g_1,g_2]^n,g_3],$$ which is trivial, by the previous paragraph. This shows that all generators of $w(G)$ have finite orders. Since $w(G)$ is torsion-free, the result follows.
\end{proof}

\section{Proof of Theorem \ref{mcw, non-commutator}}

We will require the following lemma to prove Theorem \ref{mcw, non-commutator} (a proof can be found for example in \cite[Lemma 4.1]{S2}). 
\begin{lemma}\label{bbb} The quotient $G/u(G)$ is soluble for any multilinear commutator word $u$ and a group $G$.
\end{lemma}

Some important properties of the verbal subgroup associated to a multilinear commutator word in a soluble group are given in the next lemma (taken from \cite[Theorem B]{femo}).

\begin{lemma}\label{fernandez-alcober morigi}
Let $u$ be a multilinear commutator word, and let $G$ be a soluble group. There exists a finite length series of subgroups from 1 to $u(G)$ such that:

(i) all subgroups of the series are normal in $G$;

(ii) every section $U_{i+1}/U_i$ of the series is abelian and can be generated by elements all of whose powers are $U$-values in $G/U_i$.
\end{lemma}

For the reader's convenience we restate Theorem \ref{mcw, non-commutator}.
\medskip

{\sc Theorem \ref{mcw, non-commutator}. }{\it Let $u$ be a multilinear commutator word and $v$ a non-commutator word. The word $[u,v]$ is concise.}
\medskip

It will be convenient to prove the theorem first in the particular case where $G$ is soluble. 

\begin{lemma}\label{[mcw, non-commutator] sol case} Let $u,v$ be as in Theorem \ref{mcw, non-commutator}, and set $w=[u,v]$. Let $G$ be a soluble group in which the word $w$ has only  finitely many values. Then $w(G)$ is finite. 
\end{lemma}
\begin{proof} As before, without loss of generality we assume that $w(G)$ is abelian and torsion-free.
Let $1=U_0\leq\dots\leq U_t=u(G)$ be a series as in Lemma \ref{fernandez-alcober morigi}. We assume that the series is chosen with $t$ as small as possible and argue by induction on $t$. If $t=0$, there is nothing to prove so we assume that $t\geq1$.

Let $U=U_1$, and let $X$ be the set of $u$-values contained in $U$ such that their powers are again $u$-values. Thus, $U=\langle X\rangle$. Choose $b\in G_v$. The subgroup $[U,b]$ is generated by commutators $[a,b]$, where $a\in X$. Since $U$ is normal abelian, $[a,b]^r = [a^r,b]$ for all $a\in X$ and all $r\in\Bbb Z$. Taking into account that $X$ is a set of $u$-values whose powers are again $u$-values, the cyclic subgroup $\< [a,b] \>$ is completely contained in $G_w$, which is finite. Note that all such generators $[a,b]$ are $w$-values. Hence, all generators of $[U,b]$ have finite order. Since $w(G)$ is torsion-free, conclude that $[U,b]$=1. Recall that here $b$ is an arbitrary $v$-value and therefore $[U,v(G)]=1$.  

By induction, the image of $w(G)$ in $G/U$ is finite. Therefore $w(G)^j \leq U$ for some $j$. We deduce that $[w(G)^j,v(G)] = 1$. Now the result is immediate from Lemma \ref{first reduction argument}. 
\end{proof}

Now we can proceed to the general (possibly insoluble) case. 

\begin{proof}[Proof of Theorem \ref{mcw, non-commutator}] Again, let $w=[u,v]$ and $G$ be a  group in which the word $w$ has only finitely many values. As before, without loss of generality we assume that $w(G)$ is abelian and torsion-free. Set $C = C_G(v(G))$. It is sufficient to prove that the image of $w(G)$ in $G/C$ is finite because in this case $[w(G)^j, v(G)] = 1$ for some natural $j$, whence the result is immediate from Lemma \ref{first reduction argument}.

A combination of Lemmas \ref{non-commutator > x^n} and \ref{[u(G),G^n] = 1} tells us that there is an $n$ such that the subgroup $G^n$ is contained in $v(G)$ and, moreover, $[u(G),G^n] = 1$. If $a$ is any $u$-value in $G$ and $b$ is any element  of $w(G)$, we have $[a,b]^{n} = [a,b^{n}]=1$. Since $w(G)$ is torsion-free, $[u(G),w(G)] = 1$. 

Thus, $[u(G),[u(G),v(G)]] = 1$. The Three-Subgroup Lemma ensures that $[u(G), u(G), v(G)] = 1$ and so $u(G)'\leq C$. Using Lemma \ref{bbb} we conclude that $G/C$ is soluble. Therefore, by virtue of Lemma \ref{[mcw, non-commutator] sol case}, the image of $w(G)$ in $G/C$ is finite. This completes the proof.
\end{proof}

 \section{Proof of Theorem \ref{[concise, non-commutator]}} 
The goal of this section is to prove Theorem \ref{[concise, non-commutator]}, which
we restate here for the reader’s convenience:

\medskip 
{\sc Theorem \ref{[concise, non-commutator]}. }{\it Let $u$ be a word which is (boundedly) concise in profinite groups and $v$ a non-commutator word. Then also the word $[u,v]$ is (boundedly) concise in profinite groups.}
\medskip

Start with a useful lemma on abstract groups.
\begin{lemma}\label{fg aut group}
Let $G$ be a group, $X$ a subset of $G$, and let $A=\langle \alpha_1,\dots ,\alpha_s\rangle$ be a finitely generated group of automorphisms of $G$. If the set of commutators $[x,\alpha_i]$ with $x\in X$ and $1\leq i\leq s$ is finite with $t$ elements, then $X$ lies in a union of at most $t^s$ right cosets of $C_{G}(A)$.
\end{lemma}
\begin{proof} For each $i=1,\dots,s$ set $Y_i = \{[x, \alpha_i]: x \in X\}$. Let $Y=Y_1\times\dots\times Y_s$, and let $\pi$ be the mapping of $X$ into $Y$ which takes every $x\in X$ to $([x, \alpha_1], \dots, [x, \alpha_s])$. If $x$ and $y$ belong to $X$, we have $\pi(x)=\pi(y)$ if and only if $xy^{-1}$ is fixed by each $\alpha_i$, i.e., if and only if $C_G(A)x=C_G(A)y$. Hence the elements of $X$ must lie in at most $|Y|\leq t^s$ distinct cosets of $C_G(A)$.
\end{proof}

A proof of the next lemma can be found in \cite[p. 103]{robinson-finiteness-I}. 
\begin{lemma}\label{condition for finite exponent}
Let $M,N$ be normal subgroups of a group $G$ such that both $[M:C_M(N)]=l_1$ and $[N:C_N(M)]=l_2$ are finite. Then $[M,N]$ has finite exponent dividing $l_1l_2$. 
\end{lemma} 

In what follows by ``a subgroup of a profinite group $G$'' we will always mean a {\it closed} subgroup. We say that a profinite group is generated by a set $X$ if it is {\it topologically} generated by that set. In particular, if $w$ is a group-word, $w(G)$ stands for the minimal closed subgroup containing $G_w$.

Recall that a word $w$ is said to have finite width $s$ in a group $G$ if every element of $w(G)$ can be written as a product of at most $s$ $w$-values and their inverses. 

\begin{proof}[Proof of Theorem \ref{[concise, non-commutator]}] We will give a detailed proof of the theorem only in the case where the word $u$ is boundedly concise in profinite groups. The other case can be obtained in a similar way. 

Set $w=[u,v]$ and let $G$ be a profinite group such that $|G_w|=m$. Suppose that $w$ involves exactly $r$ distinct variables. There exists a subgroup $H$ of $G$, generated by at most $mr$ elements, such that $G_w=H_w$. It is sufficient to show that $w(H)$ has finite $(m,w)$-bounded order so without loss of generality we can assume that $G=H$. By Lemma \ref{non-commutator > x^n} there is a $v$-bounded number $n$ such that $G^n \leq v(G)$. According to the solution of the Restricted Burnside Problem \cite{ze1,ze2} the quotient $G/G^n$ has finite $(m,w)$-bounded order. So the index $[G:v(G)]$ is $(m,w)$-bounded. It follows that $v(G)$ can be generated by an $(m,w)$-bounded number of elements. Moreover, the word $v$ has $(m,w)$-bounded width in $G$ (see \cite{ns1}) -- any element of $v(G)$ can be written as a product of $(m,w)$-boundedly many $v$-values. Therefore $v(G)$ can be generated by an $(m,w)$-bounded number of elements from $G_v$, say $a_1,\dots,a_s$. 

Set $C_1=C_G(u(G))$ and $C_2=C_G(v(G))$. The subgroup $v(G)$ naturally acts on $u(G)$ via conjugation. Let $A$ be the group of automorphisms of $u(G)$ induced by the action of $v(G)$.
Set $X=G_u$. Note that the set of $[x,a_i]$, where $x\in X$ and $1\leq i\leq s$, is contained in $G_w$ and so it has at most $m$ elements. By Lemma \ref{fg aut group} the set $X$ is contained in a union of at most $m^{s}$ cosets of $C_2 \cap u(G)$. Therefore in the group $G/C_2$ the word $u$ takes only $(m,w)$-boundedly many values.

As the word $u$ is boundedly concise in the class of profinite groups, we deduce that the subgroup $u(G)C_2/C_2$ is finite of $(m,w)$-bounded order. Therefore the group of automorphisms of $v(G)$ induced by conjugation by elements of $u(G)$ is finite. Repeating the above argument we deduce from Lemma \ref{fg aut group} that the quotient $v(G)C_1/C_1$ is also finite of $(m,w)$-bounded order, since $v$ is also boundedly concise.

Now, we have the following situation: $C_2\cap u(G)$ is a normal subgroup of $G$, having $(m,w)$-bounded index in $u(G)$ and centralizing $v(G)$. Also, $C_1\cap v(G)$ is a normal subgroup of $G$ that has $(m,w)$-bounded index in $v(G)$ and centralizes $u(G)$. Lemma \ref{condition for finite exponent} says that the exponent of $[u(G),v(G)]$ is $(m,w)$-bounded. This means that $w(G)$ has $(m,w)$-bounded exponent. Recall that the order of $w(G)'$ is $m$-bounded. Clearly, this implies that the order of $w(G)$ is $m$-bounded, as required.
\end{proof}

\end{document}